\documentclass[leqno]{amsart}
\usepackage{amsmath,amsfonts,amssymb,amsthm,amscd,latexsym}
\usepackage{color}
\usepackage{enumerate}
\numberwithin{equation}{section}

\newcommand{\cF}{{\mathcal F}}
\newtheorem{thm}{Theorem}
\newtheorem{theorem}[thm]{Theorem}
\newtheorem{lem}[thm]{Lemma}

\newtheorem{definition}[thm]{Definition}
\newtheorem{corollary}[thm]{Corollary}

\newtheorem{proposition}[thm]{Proposition}

\newcommand{\norm}[1]{\left\lVert#1\right\rVert}

\newcommand{\bE}{{\mathbb E}}
\newcommand{\cS}{{\mathcal S}}
\usepackage{hyperref}				
\hypersetup{colorlinks,
	linkcolor=blue,%
	citecolor=blue}	
\begin{document}

\title[Optimal range of  Haar martingale transforms]{
Optimal range  of  Haar martingale transforms and its applications
}

\author[S. Astashkin]{Sergey Astashkin}
\address[Sergey Astashkin]{Department of Mathematics, Samara National Research University, Moskovskoye
shosse 34, 443086, Samara, Russia }
\email{\textcolor[rgb]{0.00,0.00,1}{astash56@mail.ru}
 }

\author[J. Huang]{Jinghao Huang}
\address[Jinghao Huang]{School of Mathematics and Statistics, University of New South Wales,
Kensington, 2052, Australia}
\email{\textcolor[rgb]{0.00,0.00,1}{jinghao.huang@unsw.edu.au}}

\author[M. Pliev]{Marat Pliev}
\address[Marat Pliev]{Southern Institute of Mathematics (SIM)
Russian Academy of Sciences, Vladikavkaz Scientific Center
362027, Vladikavkaz
Russia and North Caucasus Center for Mathematical Research
Vladikavkaz,  362025 Russia}
\email{\textcolor[rgb]{0.00,0.00,1}{plimarat@yandex.ru}}

\author[F. Sukochev]{Fedor Sukochev}
\address[Fedor Sukochev]{School of Mathematics and Statistics, University of New South Wales,
Kensington, 2052, Australia; North-Ossetian State University, Vladikavkaz, Russia, 362025}
\email{\textcolor[rgb]{0.00,0.00,1}{f.sukochev@unsw.edu.au}}

\author[D. Zanin]{Dmitriy Zanin}
\address[Dmitriy Zanin]{School of Mathematics and Statistics, University of New South Wales,
Kensington, 2052, Australia}
\email{\textcolor[rgb]{0.00,0.00,1}{d.zanin@unsw.edu.au}}

\thanks{The work of the first author was completed as a part of the implementation of the development program of the Scientific and Educational Mathematical Center Volga Federal District, agreement no. 075-02-2022-878.}

\thanks{F. Sukochev's research is supported by the ARC}

\date{}
\begin{abstract}
Let $(\cF_n)_{n\ge 0}$  be the standard dyadic filtration on $[0,1]$.
Let $\bE_{\cF_n}$ be the conditional expectation from $ L_1=L_1[0,1]$ onto $\cF_n$, $n\ge 0$, and let $\bE_{\cF_{-1}} =0$.
We present  the sharp     estimate for the distribution function of the martingale transform $T$ defined by
\begin{align*}
Tf=\sum_{m=0}^\infty
\left(  \mathbb{E}_{\mathcal{F}_{2m}}
f-\mathbb{E}_{\mathcal{F}_{2m-1}}f \right), ~f\in L_1,
\end{align*}
in terms of the classical Calder\'{o}n operator.
  As an application, for a given symmetric function space $E$ on $[0,1]$, we identify the symmetric space $\mathcal{S}_E$,  the optimal Banach symmetric range of martingale transforms/Haar basis projections acting on $E$.
\end{abstract}

\maketitle

\section{Introduction}

Recall that the Haar system is formed by the functions $h_{0,0}(t)=h_{1}(t)=1$,
\[
 h_{n,k}(t)=h_{2^n+k}(t)=\left\{
 \begin{array}{l}
  1, t\in \Delta_{n+1}^{2k-1}\\
  -1, t\in \Delta_{n+1}^{2k}\\
  0, \text{ for all other }t\in[0,1],
 \end{array}
\right.
\]
where $n=0,1,\ldots$, $k=1,\dots,2^n$ and $\Delta_{m}^{j}=((j-1)2^{-m}, j2^{-m})$, $m=1,2,\dots$, $j=1,\dots,2^m$. It is well known  (see e.g., \cite[Ch.~3]{KS} or \cite[Proposition II.2.c.1]{LT})  that this system is a basis in $ L_p=L_p[0,1]$ for all $1\le p<\infty$ and even in every separable symmetric function space \cite[Proposition II.2.c.1]{LT}. Moreover, according to a  remarkable result due to Paley \cite{Paley} (see also \cite[Theorem II.2.c.5.]{LT} or \cite[\S\,3.3]{KS}):
\begin{quote}
  The Haar system $\{h_n\}_{n=1}^\infty $ is an unconditional basis in $L_{p}$ for every
 $1<p<\infty $.
 \end{quote}
This result turned to be extremely rich in its connections with many important problems of interest in analysis and probability theory. In particular, it served as the starting point for in-depth research undertaken by Burkholder, who has obtained sharp inequalities of Paley type for general classes of martingale transforms (see \cite{Bur, Burkholder-82, Burkholder-84}).

Nowadays, martingale transforms  provides insights  not only into probability and statistics  but also into harmonic analysis,   geometry of  various classes of Banach spaces, operator algebras and  mathematical physics (see e.g. \cite{Banuelos,Burkholder-01, Francia,HNVW} and references therein).
It is worth to note that the properties of the transformed martingales differ markedly from those of initial martingales (see e.g. the remark at the very beginning of \cite{Burkholder-84},   ``... there do exist small martingales with large  transforms''). The main result of this paper, the sharp  estimate for the distribution function of the Haar martingale transform in terms of the classical Calder\'{o}n operator,  indicates that Burkholder's remark remains relevant  also in this setting.

 We detail now our setting.
   Let $(\cF_n)_{n\ge 0}$  be the standard dyadic filtration  on $[0,1]$. Given an arbitrary sequence $\epsilon=\{\epsilon_n\}_{n\geq0},$ with $\epsilon_n\in\{-1,0,1\},$ $n\geq0,$ we consider a class of special martingale transforms  $T_\epsilon$ defined by
\begin{align}\label{T_espilon}
T_{\epsilon}x=\sum_{n\geq0}\epsilon_n \cdot (\mathbb{E}_{n}x-\mathbb{E}_{n-1 }x),~x\in L_1,
\end{align}
where $\mathbb{E}_{n}$ is the conditional expectation from $L_1$ onto $\cF_n,$ $n\ge 0,$ and the series are understood in the sense of convergence in measure.

The   classical  Calder\'{o}n operator $S$  is defined by
\begin{align}\label{Calder}
(Sx)(t):= \frac{1}{t }\int_0^t x(s)\, ds +\int_t^{1}\frac{x(s)}{s}\,ds, ~x\in  L_1 .
\end{align}

 Our main interest in this paper lies in the comparison of the distribution functions of elements $|T_\epsilon x|$ and $|S(x)|$, or equivalently,  of their decreasing right-continuous rearrangements $\mu(T_{\epsilon}x)$ and $\mu(Sx),$ respectively.
Any martingale transform $T_\epsilon$ of the form  \eqref{T_espilon} is a contraction in $L_2$  and is of   weak type $(1,1)$ with constant $2$ (this can be derived from \cite[Theorem~3.3.7]{KS} or \cite[Theorem~5.1]{NS}). Moreover, it is self-adjoint in the sense that
$$\int_0^1 T_\epsilon x(s)y(s)\,ds= \int_0^1 x(s) T_\epsilon y(s)\,ds,\quad x,y\in L_2.$$
 Therefore, $T_\epsilon$  has an upper pointwise estimate given by the operator $S$: there exists a constant $C_{{\rm abs}}$ such that
\begin{align}\label{above estimate}
\mu(T_{\epsilon}x)\leq C_{{\rm abs}}S\mu(x), ~\forall \epsilon=\{\epsilon_m\}_{m\geq0}\;\mbox{and}\;~\forall x\in L_1
\end{align}
(see e.g. \cite[Appendix]{cal}, \cite[Proposition 5.2.2, p.~50]{jm}, \cite{semen}, \cite[Example 4.15]{AstMil21} and  \cite{JSWZ}, \cite{STZ} for a more general setting).
The estimates  of the type \eqref{above estimate} are well known not only for transforms $T_{\epsilon}$  but also for other classical operators such as   the Hilbert transform and the conjugate-function operator  (see, for instance, \cite[\S\,2]{Sh-80} and \cite[Theorem~3.6.10]{BSh}), which, in fact, admit a converse.   However, the case of the converse estimate for martingale transforms remains open, and we state it here as follows
\begin{quote}
\begin{center}
\emph{Is estimate \eqref{above estimate}   optimal?}
\end{center}
\end{quote}
A similar question was also raised in \cite{HSZ} in the  non-commutative setting.
The main result of the present paper not only answers this question in the affirmative, but it also shows that the required optimality is achieved in fact  by just {\it one} operator $T=T_{\epsilon}$ with $\epsilon=\{1,0,1,0,\cdots\}_{m\geq0},$ i.e.
\begin{align}\label{T-0}
T x:=\sum_{m\geq0}    (\mathbb{E}_{{2m}}x-\mathbb{E}_{2m-1 }x),~x\in L_1.
\end{align}

  In order to state the main result, Theorem \ref{main thm} below, we recall that the dilation operator $\sigma_s$, $s>0$ (on the linear space of all measurable functions on $[0,1]$) is defined by
$\sigma_s x (t) = x(t/s) \chi_{_{[0,1]}} (t/s)$, $t\in [0,1]$.


\begin{theorem}\label{main thm}
For every function  $x\in L_1$ there exists $f\in L_1$ such that
$$
|f|\leq   3\sigma_4\mu(x)\;\;\mbox{and}\;\;\sigma_{\frac{1}{8}}S\mu(x) \le 12 \mu(Tf).$$
\end{theorem}



 Recall that Paley's result for $L_p$-spaces was later extended  (see e.g. \cite[II.2.c]{LT} or \cite[Theorem~II.9.6]{KPS}) to the setting of separable symmetric function spaces having non-trivial Boyd indices (equivalently, separable interpolation spaces between $L_p$ and $L_q $, for some $1<p\leq q<\infty$ \cite{LT}). The same condition is equivalent to the boundedness of the Calder\'{o}n operator $S$ on a symmetric function space (see e.g. \cite[Chapter 3, Theorem 6.10 and Corollary 6.11]{BSh}). Therefore, an immediate consequence of Theorem \ref{main thm} is the fact that  the unconditionality of the Haar basis in a symmetric function space $E $ can be equivalently restated in the terms of the boundedness of the operator $T$ in $E$ (cf.  \cite{LSU}).
 Moreover, this result allows us to identify the optimal Banach symmetric range of martingale transforms on $E$, for any given symmetric function space $E$ on $[0,1]$.

%
%

In Section \ref{S}, we introduce the least receptacle
$\mathcal{S}_E$ of  the Calder\'{o}n operator $S$ on a quasi-Banach symmetric function space $E$ such that $E\subset L_1$ and, as an application of Theorem \ref{main thm},
we show that the optimal symmetric quasi-Banach  range of $T$ on such a space $E $ coincides with $\mathcal{S}_E$ (see Theorem \ref{optimal01} and Corollary  \ref{opt-gen}). Moreover, in Section \ref{A1}, we prove the following:

 \begin{corollary}\label{maincor}
Assume  that  $E \subset L_1 $  and $F$ are quasi-Banach symmetric function spaces on $(0,1)$.  The following statements are equivalent:
\begin{enumerate}
\item The martingale transform $T$ is bounded from $E$ into $F$.
\item The Hilbert transform\footnote{As usual, the Hilbert transform $H$ on $[0,1]$ is defined by the  principal-value integral:
 $$
 Hx(t):=\lim_{\delta\to 0}\int_{|t-s|\ge\delta}\frac{x(s)}{t-s}\,ds,\;\;t\in  [0,1]$$
(equivalently, in this context we can consider the conjugate-function operator $x\mapsto \tilde{x}$, see e.g. \cite[p.~160]{BSh}).} $H$ is bounded from $E$ into $F$.
\item The Calder\'{o}n operator $S$ is bounded from $E$ into $F $.
\end{enumerate}
If, in addition,  $E $ is separable, then each of the statements (1) --- (3) is equivalent to the following:
\begin{enumerate}
\item[(4)] The projections $P_A:L_2\to L_2,$ $A\subset\mathbb{N},$ defined by setting
$$P_Ah_i=
\begin{cases}
h_i,& i\in A\\
0,& i\notin A
\end{cases}.
$$
extend to bounded linear mappings from $E$ into $F.$ Moreover,
$$\sup_{A\subset\mathbb{N}}\left\|P_A\right\|_{E\to F}<\infty.$$
\end{enumerate}
 \end{corollary}
  In the case when $E$  has non-trivial Boyd indices, the space $\mathcal{S}_E$ coincides with $E$ and from this angle, the result of Corollary \ref{maincor} extends and complements classical results of Paley and others cited above.

In the special case, when the space $E$ is a Lorentz space $E=\Lambda_\phi(0,1)$ with $\phi$ satisfying some natural  conditions., we provide a precise identification of the space $\mathcal{S}_E$ as another Lorentz space
 (see Section \ref{L} and Corollary \ref{main th2}).

In conclusion, we apply our results to the theory of narrow operators (see \cite{Popov-Randri, Popov-2011}). In Section~\ref{s4}, we show that the identity operator  on every separable quasi-Banach symmetric function space $E $ is a sum of two narrow operators (given by basis projections with respect to the Haar basis) bounded from $E$ into $\mathcal{S}_E$. This application extends the known result  (see \cite{Popov-Randri,Plichko-Popov,Popov-2011}) that the identity operator on a separable symmetric space $E$  with non-trivial Boyd indices is a sum of two narrow operators bounded in $
E$, which plays an important role in the theory of narrow operators.



\section{Preliminaries}

\subsection{Decreasing Rearrangement}\label{s number section}

Let $(I,m)$ denote the measure space $I = (0,1)$
equipped with the Lebesgue measure $m.$ Denote by $S(0,1)$ the space of all measurable real-valued functions on $(I,m)$ (more precisely, classes of  functions which coincide almost everywhere).

For $x\in S(0,1)$, we denote by
$\mu(x)=\mu(t;x)$ the decreasing right-continuous rearrangement of the function $|x|$ (see e.g. \cite[II, p.~117]{LT} or \cite[p.~29]{BSh}), that is,
$$\mu(t;x):=\inf\left\{
s\geq0:\ m(\{u\in [0,1]:\,|x(u)|>s\})\leq t
\right\},\quad t\in I.$$



\subsection{Symmetric (Quasi-)Banach Function Spaces}
For the general theory of symmetric Banach function spaces (resp. quasi-Banach  spaces), we refer the reader to \cite{BSh,KPS,LT} (resp. to \cite{KPR}).

\begin{definition}\label{Sym} We say that a (quasi-)normed space $(E ,\left\|\cdot\right\|_{E})  $ is a symmetric (quasi-)normed function
space
on $[0,1]$ if the following hold:
\begin{enumerate}[{\rm (a)}]
\item $E$ is a subset of $S(0,1);$
\item If $x\in  E$ and if $y\in S (0,1)$  are such that $|y|\leq|x|,$ then $y\in E$ and $\left\|y\right\|_E\leq \left\|x\right\|_E;$
\item If $x\in E$ and if $y\in S(0,1)$ are such that $\mu(y)=\mu(x),$ then $y\in E$ and $\left\|y\right\|_E=\left\|x\right\|_E.$
\end{enumerate}
If, in addition,  $(E,\left\|\cdot\right\|_E)$ is a  (quasi-)Banach space, then  $(E ,\left\|\cdot\right\|_{E})$ is called a symmetric (quasi-)Banach function space.
\end{definition}

For each $s>0$, the dilation operator $\sigma_s$ given by
$\sigma_s x (t) = x(t/s) \chi_{_{[0,1]}} (t/s)$, $t\in [0,1]$,
is well defined and bounded on every (quasi-)Banach symmetric function space $E$.

The Boyd indices \cite{LT,KPS} of a Banach symmetric function space $E$ are defined by
$$\alpha(E) =\lim_{s\to 0} \frac{\ln \norm{\sigma_s}_{E\to E}}{\ln s},~ \beta(E) =\lim_{s\to \infty } \frac{\ln \norm{\sigma_s}_{E\to E}}{\ln s}. $$
In general, $0\le \alpha(E)\le \beta(E)\le 1$.

\subsection{Calder\'{o}n operator}


The classical Hardy (or Cesaro) operator $C$ and its (formal) dual $C^*$\footnote{For any $x,y\in L_2$, we have $\int_0^1 (Cx)(s)y(s)\,ds =\int_0^1 x(s )(C^*y)(s)\,ds$.} are defined by setting
$$
(Cx)(s) := \frac1s \int_0^s x(u)\,du$$
and
$$ (C^*x)(s) :=\int_s^1\frac{x(u)}{u}du,$$
respectively \cite{HLP}. It is well known that $C:\,L_1\to L_{1,\infty}$ and $C^*:L_1\to L_1$, where the quasi-Banach symmetric space $L_{1,\infty}:=L_{1,\infty}(0,1)$ consists of all functions $x\in S(0,1)$ such that the quasi-norm
$$
\left\|
x\right\|_{L_{1,\infty}}:=\sup_{0<t\le 1}t\mu(t;x)$$
is finite.

One can easily see that the Calder\'{o}n operator $S$ (see \eqref{Calder}) satisfies the following equality
\begin{align*}
(Sx)(t) = (Cx)(t)+(C^*x)(t), ~x\in  L_1.
\end{align*}




\section{Proof of Theorem \ref{main thm}}
Let $\{\mathcal{F}_n\}_{n\ge 0}$ be the standard dyadic filtration on $[0,1]$.
Let $\bE_n$ be the conditional expectation from $L_1$ onto $\cF_n$, $n\ge 0$,
and assume for convenience that $\bE_{-1}=0.$

Also, we denote $I_n:=(2^{-n-1},2^{-n}),$ $J_n:=(0,2^{-n}),$ $n\geq0$,
and set
\begin{equation}\label{E definition}
E=\bigcup_{n\geq0}I_{2n}.
\end{equation}

Recall (see \eqref{T-0}) that the martingale transform $T$ is defined by the formula
$$Tf=\sum_{\substack{m\geq0\\ m\mbox{ \tiny is even}}}(\bE_{m} f -\bE_{m-1} f).$$

\subsection{Pointwise upper estimate: the case of the operator $C^{\ast}$}
In this subsection, we were inspired by the proof of Theorem 1 in \cite{LSU}; see also  \cite[Chapter 13.2]{Astashkin20}.

For a measurable function $x\in L_1,$ we define a function $f_1$ by setting
\begin{equation}\label{f1 definition}
f_1=\sum_{n=0}^\infty (-1)^{n+1}\mu\left(2^{-n-1};x\right)h _{n,1},
\end{equation}
where $\{h_{n,1}\}_{n\ge 0}  $ is a  subsequence of the Haar system $\{h_{n,k}\}$.  Note that
\begin{align}\label{Hndef}
h_{n,1}:=\chi_{_{J_{n+1}}}-\chi_{_{I_{n}}}  = \chi_{_{(0,2^{-n-1})}}-
\chi_{_{(2^{-n-1}, 2^{-n})}},\quad n\geq 0.
\end{align}

 The  following proposition delivers a pointwise upper estimate for an element $C^{\ast}(\mu(x))$ in terms of the operator $T$ and the function $f_1$ introduced above.

\begin{proposition}\label{cast main lemma} Let
$x\in L_1$.
If $E$ and $f_1$ are as in \eqref{E definition} and \eqref{f1 definition}, respectively, then
$$(Tf_1)\chi_{_E}\geq \frac1{2\log(2)}\chi_{_E}\cdot \sigma_{\frac12}C^{\ast}\mu(x).$$
\end{proposition}

We split the proof of Proposition \ref{cast main lemma} into several steps.

The first lemma is just a simple observation. We provide a short proof for the reader's convenience.
\begin{lem}\label{easy compute lemma} Let $a_n\in \mathbb{R}$, $n \ge 0$.
We have
$$\sum_{n=0}^\infty a_n h _{n,1}=-a_0\chi_{_{I_0}}+\sum_{m= 1}^\infty \left(
\left(\sum_{n=0}^{m-1}a_n \right)-a_{m}
\right)\chi_{_{I_m}}.$$
\end{lem}
\begin{proof}
By the definitions of $h _{n,1}$,  $I_n$ and $J_n$, we
 have
\begin{align*}
\sum_{n=0}^\infty a_n h _{n,1}&\stackrel{\eqref{Hndef}}{=}\sum_{n=0}^\infty a_n \left(
\chi_{_{J_{n+1}}}- \chi_{_{I_{n}}} \right) \\
&~=~\sum_{n=0}^\infty a_n\chi_{_{J_{n+1}} }- \sum_{n=0}^\infty a_n\chi_{_{I_{n}}}\\
&~=~\sum_{n=0}^\infty a_n\sum_{m= n+1 }^\infty \chi_{_{I_m}}-\sum_{m=0}^\infty  a_{m}\chi_{_{I_m}}\\
&~=~\sum_{m= 1}^\infty \chi_{_{I_m}}\sum_{n =0 }^{  m-1 }a_n-\sum_{m=0}^\infty a_{m}\chi_{_{I_m}}\\
& ~=~-a_0\chi_{_{I_0}}+\sum_{m= 1}^\infty \left(
\left(\sum_{n=0}^{  m-1}a_n\right) -a_{m}
\right)\chi_{_{I_m}} .
\end{align*}
This completes the proof.
\end{proof}

 For the sake of convenience, we observe   the following standard result.

\begin{lem}\label{leibniz} Let $\{b_n\}_{n\geq1}\subset \mathbb{R}$ be a sequence with alternating signs and with increasing absolute values. We have
$$\left|\left(\sum_{n=1}^{m-1}b_n\right)-b_m \right|\leq 2|b_m|.$$
\end{lem}
%
%
%

\begin{lem}\label{f1 above lemma} Let $x\in L_1$.
If $f_1$ is as in \eqref{f1 definition}, then we have
$$|f_1|\leq 2\sigma_2\mu(x).$$
\end{lem}
\begin{proof} Let
$$b_n:=(-1)^{n+1}\mu\left(2^{-n-1};x\right),\quad  n\ge 1.$$
This is a sequence with alternating signs and with increasing absolute values. By the definition of $f_1,$ we have
$$f_1=\sum_{n=1}^{\infty}b_nh_{n,1}.$$
By Lemma \ref{easy compute lemma}, we have
$$f_1|_{I_m}=\big(\sum_{n=1}^{m-1}b_n\big)-b_m,\quad m\geq 1.$$
By Lemma \ref{leibniz}, we have
$$|f_1|\Big|_{I_m}\leq 2|b_m|=2\mu(2^{-m-1},x)\leq 2\sigma_2\mu(x)\Big|_{I_m},\quad m\geq 1.$$
A combination of these inequalities yields $|f_1|\leq 2\sigma_2\mu(x)$ on every $I_m,$ $m\geq1,$ and, therefore, on $(0,\frac12).$ On the interval $(\frac12,1)$,
we have
$$\left|
f_1\chi_{_{(\frac12,1)}}\right |\stackrel{\eqref{f1 definition}}{=}   \mu\left(
0 ; x
\right) \le  \mu(x)\chi_{_{(\frac12,1)}}.$$
This completes the proof.
\end{proof}

\begin{proof}[Proof of Proposition \ref{cast main lemma}] By definitions \eqref{T-0} and \eqref{Hndef},   we have $Th_{n,1}=h _{n,1}$ for every  odd natural number $n$ and $Th_{n,1}=0$ for every even natural number $n.$ Therefore, we have
$$Tf_1=
\sum_{n=0}^\infty (-1)^{n+1}\mu\left(2^{-n-1};x\right)Th_{n,1}
 =
\sum_{\substack{n\geq1\\ n\mbox{ \tiny is odd}}}\mu\left(
2^{-n-1};x\right)h _{n,1} =\sum_{n\geq1}c_nh_{n,1},$$
where
$$c_n=
\begin{cases}
\mu\left(2^{-n-1 };x\right ),&n\mbox{   is odd};\\
0,&n\mbox{ is even}.
\end{cases}
$$
By Lemma \ref{easy compute lemma}, for even $m,$
\begin{align}\label{tfim}
Tf_1\big|_{I_m}=\left(\sum_{n=0}^{ m-1}c_n\right) -c_m=\sum_{\substack{0\leq n\leq m-1 \\ n\mbox{  \tiny  is odd}}}\mu(2^{-n-1 };x).
\end{align}
For an even natural number  $m\geq 1,$ it follows that
\begin{align}\label{1nmeven}
\sum_{\substack{0\leq n\leq m-1 \\ n\mbox{  \tiny  is odd}}}\mu(2^{-n-1 };x)&\geq \frac12\sum_{n=0}^{  m-1 }\mu(2^{-n-1 };x)\nonumber \\
&=\frac1{2\log(2)}\sum_{n=0}^{  m-1 }\int_{2^{-n-1 }}^{2^{-n}}\mu(2^{-n-1 };x)\frac{ds}{s} \nonumber\\
&\geq\frac1{2\log(2)}\sum_{n=0}^{  m -1 }\int_{2^{-n-1 }}^{2^{-n}}\mu(s; x)\frac{ds}{s}\\
&=\frac1{2\log(2)}\int_{2^{-m }}^1\mu(s;x)\frac{ds}{s}\nonumber\\
&=\frac1{2\log(2)} (C^{\ast}\mu(x))(2^{-m})\nonumber.
\end{align}
By \eqref{tfim} and \eqref{1nmeven}, we have
$$Tf_1\big|_{I_m}\geq \frac1{2\log(2)}\left( \sigma_{\frac12}C^{\ast}\mu(x) \right)\Big |_{I_m},\quad m\geq1\mbox{    is even}.$$
This completes the proof.
\end{proof}

\subsection{Pointwise upper estimate: the case of the operator $C$}
 As above, we denote $I_n=(2^{-n-1},2^{-n}),$ $J_n=(0,2^{-n}),$ $n\geq0.$ For any integer $n\geq0$, we define the function $g_n$ by setting
\begin{equation}\label{gn definition}
g_n: =\sum_{k=0}^{\infty}2^{-(n-k)_+}\cdot  \chi_{_{I_k}},
\end{equation}
where $u_+ = \begin{cases}u,~&\mbox{if $u\ge 0$};\\
0,~&\mbox{if $u<  0$}.
\end{cases} $

For any $x\in L_1,$ we define the function $f_2$ by
\begin{equation}\label{f2 definition}
f_2: =\sum_{\substack{n\geq0\\ n\mbox{  \tiny  is even}}}\mu\left(
2^{-n-2};x
\right)
\chi_{_{I_n}}.
\end{equation}

Now, we state the main result of this subsection.

\begin{proposition}\label{c main lemma}
 Let
$x\in L_1$. If $E$ and $f_2$ are as in  \eqref{E definition} and \eqref{f2 definition},  respectively, then
$$(Tf_2)\cdot\chi_{_E}\geq\frac16\chi_{_E} \cdot C\mu(x).$$
\end{proposition}

We split the proof of Proposition~\ref{c main lemma} into several steps.

\begin{lem}\label{f2 above lemma}Let
$x\in L_1$.  We have
$$|f_2|\leq \sigma_4\mu(x).$$
\end{lem}
\begin{proof} Observe that
$$f_2\stackrel{\eqref{f2 definition}}{=}\sum_{\substack{n\geq0\\ n\mbox{   \tiny  is even}}}\mu(2^{-n-2};x)\chi_{_{I_n}} \leq \sum_{\substack{n\geq0\\ n\mbox{  \tiny  is even}}}\mu(2^{-n-2};x)(\chi_{_{I_n}}+\chi_{_{I_{n+1}}})\leq\sigma_4\mu(x).$$
\end{proof}

\begin{lem}\label{c first compute lemma}Let $n\ge 0$ be an even number.  If $E$ and $g_n$ are as  in \eqref{E definition} and \eqref{gn definition}, respectively, then
$$(T\chi_{_{I_n}})\chi_{_E}\geq \frac13g_n\chi_{_E} ,$$
where $I_n=(2^{-n-1},2^{-n})$.
\end{lem}
\begin{proof} Let $\epsilon=\{(-1)^n\}_{n\geq0}$ and consider the operator $T_{\epsilon}$ as in \eqref{T_espilon}, i.e.
$$T_{\epsilon}f=\sum_{m\geq0}(-1)^m(\bE_mf-\bE_{m-1}f),\quad f\in L_1.$$
Clearly,
\begin{align}\label{bemin}
\bE_m\chi_{_{I_n}}=
\begin{cases}
\chi_{_{I_n}},& m\geq n+1;\\
2^{m-1-n}\chi_{_{J_m}},& m\leq n,
\end{cases}
\end{align}
and hence
$$\bE_m\chi_{_{I_n}}=\bE_{m-1}\chi_{_{I_n}},\quad m\geq n+2.$$
Since $\bE_{-1}=0$, it follows that
\begin{align}\label{SXIN}
T_{\epsilon}\chi_{_{I_n}}&~=~\sum_{m=0}^{n+1}(-1)^m(\bE_m\chi_{_{I_n}}
-\bE_{m-1}\chi_{_{I_n}})\nonumber\\
&~=~\sum_{m=0}^{n+1}(-1)^m\bE_m\chi_{_{I_n}}-\sum_{m=0}^n(-1)^{m-1}\bE_m\chi_{_{I_n}}\\
&\stackrel{\eqref{bemin}}{=}(-1)^{n+1}\chi_{_{I_n}}+2\sum_{m=0}^n(-1)^m2^{m-1-n}\chi_{_{J_m}}.\nonumber
\end{align}
By the definition of $J_m$ and $I_m$, we have
$$\chi_{_{J_m}}=\sum_{k= m}^\infty \chi_{_{I_k}}.$$
Therefore,
\begin{align}\label{m0n1m2}
\sum_{m=0}^n(-1)^m2^{m-1-n}\chi_{_{J_m}}&=\sum_{m=0}^n(-1)^m2^{m-1-n}\sum_{k= m}^\infty  \chi_{_{I_k}}\nonumber \\
&=\sum_{k=0}^\infty \chi_{_{I_k}}\sum_{m=0}^{\min\{k,n\}}(-1)^m2^{m-1-n}\\
&=2^{-n-1}\sum_{k=0}^{\infty }\chi_{_{I_k}}\cdot \frac{(-2)^{\min\{k,n\}+1}-1}{-3}\nonumber .
\end{align}
Now,   we arrive at
\begin{align*}
T_{\epsilon}\chi_{_{I_n}}&\stackrel{\eqref{SXIN}}{=}(-1)^{n+1}\chi_{_{I_n}}+2\sum_{m=0}^n(-1)^m2^{m-1-n}\chi_{_{J_m}}\\
&\stackrel{\eqref{m0n1m2}}{=}  (-1)^{n+1}\chi_{_{I_n}}+ 2^{-n }\sum_{k=0}^{\infty }\chi_{_{I_k}}\cdot \frac{(-2)^{\min\{k,n\}+1}-1}{-3}\\
&~= ~(-1)^{n+1}\chi_{_{I_n}} + 2^{-n}\sum_{k=0}^{n-1}\chi_{_{I_k}}\cdot \frac{(-2)^{k+1}-1}{-3}+2^{-n}\sum_{k=n}^{\infty}\chi_{_{I_k}}\cdot \frac{(-2)^{n+1}-1}{-3}.
\end{align*}
If $n$ is even, then
\begin{align}\label{SXINXIN}
T_{\epsilon}\chi_{_{I_n}}+\chi_{_{I_n}}
=2^{-n}\sum_{k=0}^{n-1}\chi_{_{I_k}}\cdot \frac{(-1)^k2^{k+1}+1}{3}+2^{-n}\sum_{k=n}^{\infty}\chi_{_{I_k}}\cdot \frac{2^{n+1}+1}{3}.
\end{align}
Thus, for any even number $n\ge 0$,  we have
\begin{align*}
(T_{\epsilon}\chi_{_{I_n}}+\chi_{_{I_n}})\cdot \chi_{_E}&\stackrel{\eqref{SXINXIN}}{=}\left(2^{-n}\sum_{k=0}^{n-1}\chi_{_{I_k}}\cdot \frac{2^{k+1}+1}{3}+\sum_{k=n}^{\infty}\chi_{_{I_k}}\cdot \frac{2+2^{-n}}{3}\right)\cdot \chi_{_E}\\
&~\geq~ \left(2^{-n}\sum_{k=0}^{n-1}\chi_{_{I_k}}\cdot \frac{2^{k+1}}{3}+\sum_{k=n}^{\infty}\chi_{_{I_k}}\cdot \frac{2}{3}\right)\cdot \chi_{_E}\\
&~=~\frac23\chi_{_E}\cdot \left(\sum_{k=0}^{\infty}2^{-(n-k)_+}\cdot \chi_{_{I_k}}\right).
\end{align*}
Since $T_{\epsilon}+{\rm id}=2T$ (see \eqref{T-0}), the assertion follows.
\end{proof}

\begin{lem}\label{c second compute lemma} For every $n\geq0,$ we have $$g_n\geq \frac12C\chi_{_{J_n}},$$
where $g_n$ is defined by formula  \eqref{gn definition} and $J_n=(0,2^{-n})$.
\end{lem}
\begin{proof} If $t\in I_k,$ $k\geq n,$ then we have $g_n(t)=1$ and $(C\chi_{_{J_n}})(t)=1$.   If $t\in I_k=(2^{-k-1},2^{-k}),$ $k<n,$ then $t\notin J_n,$ and therefore,
$$(C\chi_{_{J_n}})(t)=\frac{m(J_n)}{t}=\frac1{2^nt}\leq 2^{k+1-n}=2\cdot 2^{k-n}=2g_n(t),$$
and the desired inequality follows.
\end{proof}

\begin{proof}[Proof of Proposition \ref{c main lemma}] Applying successively the definition of $f_2$ (see \eqref{f2 definition}), Lemma \ref{c first compute lemma} and Lemma \ref{c second compute lemma}, we obtain
\begin{eqnarray}
(Tf_2)\chi_{_E}&=& \sum_{\substack{n\geq0\\ n\mbox{  \tiny  is even}}}\mu(2^{-n-2}; x)(T\chi_{_{I_n}})\chi_{_E}\nonumber\\
&\geq& \frac13\sum_{\substack{n\geq0\\ n\mbox{  \tiny  is even}}}\mu(2^{-n-2};x)g_n\chi_{_E}\nonumber\\
&\geq&\frac16\chi_{_E}\cdot C\left(\sum_{\substack{n\geq0\\ n\mbox{  \tiny  is even}}}\mu(2^{-n-2}; x)\chi_{_{J_n}}\right).
\label{TFCHI}
\end{eqnarray}

Recall that $I_n=(2^{-n-1},2^{-n})$ and $J_n=(0,2^{-n}).$ Observe that
$$\sum_{\substack{n\geq0\\ n\mbox{  \tiny  is even}}}\mu(2^{-n-2};x)\chi_{_{J_n}}\geq \sum_{\substack{n\geq0\\ n\mbox{  \tiny  is even}}}\mu(2^{-n-2};x)(\chi_{_{I_n}}+\chi_{_{I_{n+1}}})\geq\mu(x),$$
which together with \eqref{TFCHI} yields the assertion.
\end{proof}

\subsection{Proof of Theorem \ref{main thm}}
Here, we complete the proof of  Theorem~\ref{main thm}, which is a simple consequence of the estimates obtained in the previous subsections.

\begin{lem}\label{final pointwise estimate}  Let
$x\in L_1$ and let  $f:=f_1+f_2,$ where
  $f_1$ and $f_2$ are defined in  \eqref{f1 definition} and \eqref{f2 definition}, respectively.
We have
$$|f|\leq 3\sigma_4\mu(x)\;\;\mbox{and}\;\;(Tf)\cdot\chi_{_E} \geq \frac16\chi_{_E}\cdot \sigma_{\frac12}S\mu(x).$$
\end{lem}
\begin{proof} By Lemmas  \ref{f1 above lemma} and \ref{f2 above lemma}, we have
$$|f|\leq |f_1|+|f_2|\leq 2\sigma_2\mu(x)+\sigma_4\mu(x)\leq 3\sigma_4\mu(x).$$
On the other hand, Propositions \ref{cast main lemma} and \ref{c main lemma} imply
\begin{align*}
(Tf)\cdot\chi_{_E}&=(Tf_1)\cdot\chi_{_E}+(Tf_2)\cdot\chi_{_E}\\
&\geq \frac1{2\log(2)}\chi_{_E}\cdot\sigma_{\frac12}C^{\ast}\mu(x)+ \frac16\chi_{_E}\cdot C\mu(x)\\
&\geq \frac16\chi_{_E}\cdot\sigma_{\frac12}C^{\ast}\mu(x)+ \frac16\chi_{_E}\cdot \sigma_{\frac12}C\mu(x)\\
&=\frac16\chi_{_E}\cdot \sigma_{\frac12}S\mu(x),
\end{align*}
and everything is done.
\end{proof}


\begin{lem}\label{y vs ychie} If $y=\mu(y)\in S(0,1),$ then
$$\frac12\sigma_{\frac14} y \le \mu(\chi_{_E}\cdot y),$$
where $E$ is defined  in \eqref{E definition}.
\end{lem}
\begin{proof} Observe that from the definition of $E$ it follows
$$\chi_{_{E^c}}\cdot y\leq \sigma_2(\chi_{_E}\cdot y).$$
Thus, by \cite[Proposition~2.1.7]{BSh}, we have
\begin{align*}
\frac12 \sigma_{\frac14} y=\frac12\sigma_{\frac14} \mu\big(\chi_{_{E^c}}\cdot y+\chi_{_E}\cdot y\big)&\leq \frac12\sigma_{\frac12} \mu(\chi_{_{E^c}}\cdot y)+ \frac12\sigma_{\frac12} \mu(\chi_{_E}\cdot y)\\
&\leq  \frac12\mu(\chi_{_E}\cdot y)+\frac12 \sigma_{\frac12} \mu(\chi_{_E}\cdot y)\leq  \mu(\chi_{_E}\cdot y),
\end{align*}
and  the proof is completed.
\end{proof}

\begin{proof}[Proof of Theorem \ref{main thm}]

Let $f$ be defined as in Lemma \ref{final pointwise estimate}. Then,  $|f|\leq 3\sigma_4\mu(x).$ Moreover, by Lemmas \ref{final pointwise estimate} and \ref{y vs ychie}, we have
$$\mu(Tf)\geq \mu((Tf)\cdot\chi_{_E})\geq \frac16\mu(\chi_{_E}\cdot\sigma_{\frac12}S\mu(x))\geq \frac1{12}\sigma_{\frac18}S\mu(x).$$
\end{proof}

\section{Applications to the geometry of  Banach spaces}
\subsection{Optimal symmetric quasi-Banach range for the
martingale transforms}\label{S}

From Theorem \ref{main thm} and estimate \eqref{above estimate} it follows that  the optimal symmetric Banach  range of the martingale transform $T$ on a quasi-Banach symmetric function space $E$ coincides with that of the Calder\'{o}n operator $S$ on $E$. Thus, we arrive at the problem of a description of the least receptacle of the operator  $S$ acting on 
$E$. To solve the latter problem, we
  employ the
  description of the optimal symmetric  range for the Calder\'{o}n operator defined on a quasi-Banach symmetric space on $(0,\infty)$ given in \cite{STZ}.

For definitions related to quasi-Banach symmetric spaces on $(0,\infty)$ (which differ only slightly from  those in the case $[0,1]$)  we refer the reader to the books \cite{BSh,KPS,LT}. In particular, $S(0,\infty )$ is the set of all measurable functions $x$ on $(0,\infty)$ such that $m(\{t : |x(t)| > s\})$ is finite for some $s > 0.$

Recall that
$$
L_{1,\infty}(0,\infty):=\{f\in S(0,\infty): \norm{f}_{L_{1,\infty}(0,\infty)}:=\sup_{t>0}t\mu(t;f)<\infty\}$$
and
$$
\Lambda_{\log}(0,\infty ):= \left\{x \in S(0,\infty ):\norm{x}_{\Lambda_{\log }(0,\infty )}:=\int_0^\infty \mu(s;x)\frac{ds}{s+1}<\infty  \right\}.$$

  The Calder\'{o}n operator (on the semiaxis) is given by
$$ (S_\infty x)(t):=\frac{1}{t}\int_0^t x(s)ds +\int_t^\infty x(s)\frac{ds}{s}, ~x\in \Lambda_{\log}(0,\infty) .$$

For convenience of the reader, we describe first shortly the main result in \cite{STZ}. Further, we still denote a symmetric function space on $[0,1]$ by $E$, while the notation $E(0,\infty)$ will be reserved for symmetric function spaces on $(0,\infty)$.

Given quasi-Banach symmetric space $E(0,\infty)$ such that  $E(0,\infty )\subset \Lambda_{\log}(0,\infty )$, we define the linear space $\cS_E(0,\infty )$ by
\begin{align}\label{ETOF2}
\cS_E(0,\infty )= \left\{x\in  (L_{1,\infty}+L_{\infty})(0,\infty ):  \exists y\in E(0,\infty ), \, \mu(x)\leq S_\infty  \mu(y) \right \},
\end{align}
equipped with the functional
$$
x\mapsto \left\|x\right\|_{\cS_{E(0,\infty )}}:=\inf\{\left\|y\right\|_{E }:\mu(x)\leq S\mu(y)\}.$$




\begin{thm}\label{quasi-banach opt range}\cite[Theorem 26]{STZ} Let $E(0,\infty )\subset \Lambda_{\log}(0,\infty  )$ be a quasi-Banach symmetric space on $(0,\infty ).$ We have
\begin{enumerate}[{\rm (i)}]
\item $(\mathcal{S}_E(0,\infty  ),\left\|\cdot\right\|_{\mathcal{S}_E}(0,\infty ))$ is a quasi-Banach symmetric function space.
\item $\mathcal{S}_E(0,\infty  )$ is the optimal symmetric quasi-Banach range for the operator $S$ on $E(0,\infty ).$
\end{enumerate}
\end{thm}

 Below, we obtain a similar identification of the optimal symmetric  range for the Calder\'{o}n operator on a given quasi-Banach symmetric space on $(0,1)$.

\begin{definition}\label{quasi-banach range}
Let $E$ be a quasi-Banach symmetric space on $(0,1)$ such that $E\subset L_1.$ Define the linear space
\begin{align}\label{ETOF}
\cS_E= \left\{x\in  L_{1,\infty}=L_{1,\infty}(0,1):\,  \exists y\in E , \, \mu(x)\leq S\mu(y) \right \},
\end{align}
and equip it with the functional
$$
x\mapsto \left\|x\right\|_{\cS_E}:=\inf\{\left\|y\right\|_{E }:\mu(x)\leq S\mu(y)\}.$$
\end{definition}


\begin{theorem}\label{optimal01} Let $E\subset L_1$ be a quasi-Banach symmetric space on $(0,1).$ We have
\begin{enumerate}[{\rm (i)}]
\item $(\mathcal{S}_E,\left\|\cdot\right\|_{S_E})$ is a quasi-Banach symmetric function space.
\item $\mathcal{S}_E $ is the optimal symmetric quasi-Banach range for the operator $S$ on $E.$
\end{enumerate}
\end{theorem}
\begin{proof}(i).  For simplicity of notations, we may assume  that $\left\|\chi_{_{(0,1)}} \right\|_{E} =1.$
Define  a symmetric quasi-Banach function space $F(0,\infty )$ on $(0,\infty)$ by setting
$$F(0,\infty ):= \left\{x\in L_1 (0,\infty ):
\left\|x\right \|_{F(0,\infty )}:=\left\|
\mu(x)\chi_{_{(0,1)}}\right\|_{E}+\left\|x\right\|_{L_1(0,\infty )}<\infty \right\}.$$

We claim that for every $x$ supported on $(0,1),$ we have
\begin{equation}\label{e vs f}
\frac14\left\|x\right\|_{\cS_F(0,\infty )}\leq \left\|x\right\|_{\cS_E}\leq 2 \left\|x\right\|_{\cS_F(0,\infty )}.
\end{equation}
Indeed, if $x\in \cS_E,$ then there exists $y\in E$ such that $$\mu(x)\leq S\mu(y) \mbox{ and }\left\|y\right\|_{E}\leq 2\left\|x\right\|_{\cS_E}.$$
Extending $y$ to a function on $(0,\infty)$ by setting $y=0$ on $(1,\infty),$ we still have $\mu(x)\leq S_\infty \mu(y)$. Moreover, in view of the embedding $E\subset L_1$ with constant $1$ (see e.g. \cite[Theorem~II.4.1]{KPS}), it holds
$$\left\|y\right\|_{F(0,\infty )}=
\left\|y\right\|_{E}+\left\|y\right\|_{L_1(0,\infty )}
\leq 2\left\|y\right\|_{E}\leq 4\left\|x\right\|_{\cS_E}.$$
Taking the infimum over all such $y,$ we infer that $\left\|x\right\|_{\cS_F(0,\infty )}\leq 4 \left\|x\right\|_{\cS_E}.$

Next, let $x\in \cS_{F(0,\infty)}$ with support in $(0,1)$ and let $y\in F(0,\infty )$ be such that $\mu(x)\leq S_\infty \mu(y)$ and $\left\|y\right\|_{F(0,\infty)}
\leq 2\left\|x\right\|_{\cS_F(0,\infty )}$ (see Theorem \ref{quasi-banach opt range}).
 Without loss of generality, we may assume that $y=\mu(y).$ Set
$$z(t)=\Big(y(t)+\int_1^{\infty}y(s)\frac{ds}{s}\Big)\chi_{_{(0,1)}}(t),\quad t\in(0,1).$$
We have
$$\mu(t;x)\leq (S_\infty y)(t) \leq (Sz)(t),\quad t\in(0,1).$$
Also,
\begin{align*}
\left\|z\right\|_{E}&\leq \left\|
y\chi_{_{(0,1)}}\right\|_{E}
+\int_1^{\infty}y(s)\frac{ds}{s}\\
&\leq \left\|y\chi_{_{(0,1)}}\right\|_{E}+\left\|y\right\|_{L_1(0,\infty )}\\
&=\|y\|_{F(0,\infty)}\leq 2\left \|x\right\|_{\cS_F(0,\infty )}.
\end{align*}
Taking the infimum over all such $z,$ we get $\left\|x\right\|_{\cS_E}\leq 2 \left\|x\right\|_{\cS_F(0,\infty )} .$

Clearly, $\left\|\cdot\right\|_{\cS_E}$ is a homogeneous functional. Since $\left\|\cdot \right\|_{\cS_F(0,\infty )}$ is a quasi-norm (see Theorem \ref{quasi-banach opt range} above), it follows from \eqref{e vs f} that $\left\|\cdot\right\|_{\cS_E}$ is also a quasi-norm.

Let us now prove the completeness of $(\cS_E,\left\|\cdot\right\|_{\cS_E}).$ Let $(x_n)_{n\geq0}$ be a Cauchy sequence in $\cS_E.$ By \eqref{e vs f}, $(x_n)_{n\geq0}$ is a Cauchy sequence in $\cS_F(0,\infty ).$ By the completeness of $\cS_F(0,\infty ),$ we have that $x_n\to x$ in $\cS_F(0,\infty ).$ Clearly, $x$ is also supported on $(0,1).$ Again using \eqref{e vs f}, we conclude that $x_n\to x$ in $\cS_E.$
On the other hand, by the definition of $\norm{\cdot}_{\cS_E}$,
we obtain that the quasi-norm $\norm{\cdot}_{\cS_E} $ is  symmetric.

(ii)
From the definition of $\cS_E$ it follows immediately that $\cS_E$ is the minimal receptacle of the operator $S$ in the category of quasi-Banach symmetric function spaces (see also \cite[p.3549]{STZ} for a full proof in the setting of $(0,\infty)$).
\end{proof}

  The following result is a combination of Theorems \ref{main thm} and  \ref{optimal01} with estimate \eqref{above estimate}.
 \begin{corollary}
 \label{opt-gen}
Assume  that   $E\subset L_1$  is   quasi-Banach symmetric function space on $(0,1)$. Then,  the space $\mathcal{S}_E$ generated by  the Calder\'{o}n operator $S$ is the optimal symmetric quasi-Banach range for the martingale transform  $T$ on $E$.
 \end{corollary}


\subsection{Optimal symmetric Banach range for the
martingale transforms in Lorentz spaces}
\label{L}
 Here, we apply the results obtained in the preceding sections to present  a description of the optimal symmetric Banach range of the
martingale transforms on Lorentz function spaces on $[0,1]$.


Let $\phi : [0,1) \to [0,1 )$ (respectively, $\phi : [0,\infty ) \to [0,\infty  )$) be  an increasing concave function such that
$\lim_{t\to 0+} \phi(t) =0$ (or briefly $\phi(+0) =0$).  The Lorentz space $\Lambda_\phi$ (respectively, $\Lambda_\phi(0,\infty  ) $ is defined by setting
$$
\Lambda _{\phi}:  = \left\{x\in S(0,1 ) : \norm{x}_{\Lambda_\phi}:=  \int_0^1 \mu(s;x)\, d\phi(s) <\infty \right\}$$
(respectively,
$$
\Lambda _{\phi} (0,\infty  )  = \left\{x\in S(0,\infty  ) : \norm{x}_{\Lambda_\phi}:=  \int_0^\infty \mu(s;x)\, d\phi(s) <\infty \right\}\,).$$


In \cite{STZ2},  the optimal Banach symmetric range of the Calder\'{o}n operator $S_\infty$ on Lorentz spaces  $\Lambda_\phi(0,\infty)$ was determined.
 Let us state their main result.
\begin{theorem}\cite[Theorem 11]{STZ2}\label{thstz2}
Let $\phi:[0,\infty ) \to [0,\infty )$ be an increasing concave function such that $\phi(0+) =0$. Suppose the function
$$
\psi(u):=\inf_{w>1} \frac{\phi(uw)}{1+\log(w)}$$
satisfies $\lim_{t\to \infty }\frac{\psi(t)}{t}=0$.
Then, the conditions $\Lambda_\phi(0,\infty ) \subset \Lambda _{\log }(0,\infty)$ and
$$\int_0^ u \frac{\psi(t)}{t}dt +u\int_u^\infty \frac{\psi(t)}{t^2}dt \le c_{\phi, \psi}\phi(u), ~u >0,$$ imply the following:
\begin{enumerate}[{\rm (i)}]
\item The Calder\'{o}n operator $S_\infty :\Lambda_{\phi}(0,\infty )\to\Lambda_{\psi}(0,\infty ) $ is bounded;
\item for every $x\in\Lambda_{\psi}(0,\infty ),$ there exists $y\in\Lambda_{\phi}(0,\infty )$ such that $\mu(x)\leq S_\infty \mu(y)$ and $\left\|y\right\|_{\Lambda_{\phi}(0,\infty )}\leq 8\left\|x\right\|_{\Lambda_{\psi}(0,\infty ) }.$
\end{enumerate}
\end{theorem}

 We apply Theorem \ref{thstz2} to obtain a similar result for Lorentz function spaces on $[0,1]$, i.e., we determine  the optimal range of the Calder\'{o}n operator $S$ on a Lorentz space $\Lambda_{\phi}$ as some  Lorentz space $\Lambda_{\psi}$.

Let $\phi:[0,1)\to[0,1)$ be an increasing concave function such that $\phi(0+)=0.$ We set
\begin{equation}\label{psi function}\psi(u):=\inf_{1< w< \frac1u }\frac{\phi(uw)}{1+\log(w)},~u\in [ 0,1).
\end{equation}
\begin{thm}\label{main th} Let $\phi:[0,1)\to[0,1)$ be an increasing concave function such that $\phi(0+)=0$ and let $\psi$ be the function defined by the formula \eqref{psi function}.
  If the inequality
\begin{equation}\label{crit}
\int_0^u\frac{\psi(t)}{t}\,dt+u\int_u^{1}\frac{\psi(t)}{t^2}\,dt\leq c_{\phi,\psi}\phi(u),\quad u\in (0,1),
\end{equation} holds for some constant $c_{\phi,\psi}$, then:
\begin{enumerate}[{\rm (i)}]
\item The Calder\'{o}n operator $S:\Lambda_{\phi}\to\Lambda_{\psi}$ is bounded;
\item for every $x\in\Lambda_{\psi},$ there exists $y\in\Lambda_{\phi}$ such that $\mu(x)\leq S\mu(y)$ and $\left\|y\right\|_{\Lambda_{\phi}}\leq 8 \left\|x\right\|_{\Lambda_{\psi}}.$
\end{enumerate}
\end{thm}
\begin{proof}Without loss of generality, we may assume that $\phi(1)=1$.
We define the functions $\tilde{\phi}$ and $\tilde{\psi}$ on $(0,\infty)$ by setting
$$
\tilde{\phi}(t):= \begin{cases}\phi(t), ~t\in (0,1)\\
1+  \log(t),~t\ge 1
\end{cases}$$
and
$$\tilde{\psi}(u):=\inf_{w>1}\frac{\phi(uw)}{1+\log(w)}. $$
If $u\ge 1$, then we have
\begin{align}\label{barpsiu}
\tilde{\psi}(u) =\inf_{w>1}\frac{\phi(uw)}{1+\log(w)} = \inf_{w>1} \frac{1+\log(u)+\log(w)}{1+\log(w)} =1 .
\end{align}
Moreover, in the case when $0< u< 1$
\begin{align*}
\tilde{\psi}(u) =\inf_{w>1}\frac{\phi(uw)}{1+\log(w)} &=\min\left\{
\inf_{1<w<\frac1u }\frac{\phi(uw)}{1+\log(w)}, ~\inf_{w\ge \frac1u}\frac{\phi(uw)}{1+\log(w)} \right\} \\
&=\min\left\{
\inf_{1<w<\frac1u }\frac{\phi(uw)}{1+\log(w)} , ~ \inf_{w\ge \frac1u}\frac{1+\log(uw)}{1+\log(w)} \right\}\\
&=\min\left\{
\inf_{1<w<\frac1u }\frac{\phi(uw)}{1+\log(w)} , ~ 1 \right\}\\
&=\min\left\{
\psi(u),~ 1 \right\}.
\end{align*}
One can easily verify that $\psi$ is increasing on $(0,1)$ (see also  the proof of Lemma 5 in \cite{STZ2}). Hence, $\psi(u)\le \psi(1)=1 $ whenever $0< u< 1$. Thus,
$$\tilde{\psi}(u)= \psi(u)\;\;\mbox{for all}\;\;0< u< 1. $$

 Next, if $0<u\le 1$, we have
 \begin{align*}
 \int_0^ u \frac{\tilde{\psi}(t)}{t}dt +u\int_u^\infty \frac{\tilde{\psi}(t)}{t^2}dt
 &~=~
 \int_0^ u \frac{\tilde{\psi}(t)}{t}dt +u\int_u^1 \frac{\tilde{\psi}(t)}{t^2}dt + u \int_1^\infty  \frac{\tilde{\psi}(t)}{t^2}dt\\
  &
 \stackrel{\eqref{crit}}{\le} c_{\phi,\psi}\phi(u)+ u \int_1^\infty  \frac{\tilde{\psi}(t)}{t^2}dt\\
&  \stackrel{\eqref{barpsiu}}{\le} c_{\phi,\psi}\phi(u)+ u \int_1^\infty  \frac{1}{t^2}dt\\
&~= ~ c_{\phi,\psi}\phi(u)+ u \\
&~\le~ (c_{\phi,\psi}+1 )\tilde{\phi}(u),
\end{align*}
 and in the case $u > 1$
 \begin{align*}
 \int_0^ u \frac{\tilde{\psi}(t)}{t}dt +u\int_u^\infty \frac{\tilde{\psi}(t)}{t^2}dt
 &~=
 \int_0^ 1 \frac{\tilde{\psi}(t)}{t}dt +\int_1^ u \frac{\tilde{\psi}(t)}{t}dt + u \int_u^\infty   \frac{\tilde{\psi}(t)}{t^2}dt\\
& \stackrel{\eqref{crit}}{\le}  c_{\phi,\psi}\phi(1 ) +\int_1^ u \frac{\tilde{\psi}(t)}{t}dt + u \int_u^\infty   \frac{\tilde{\psi}(t)}{t^2}dt   \\
& \stackrel{\eqref{barpsiu}}{\le}  c_{\phi,\psi}\phi(1 ) +\int_1^ u \frac{1}{t}dt + u \int_u^\infty   \frac{1}{t^2}dt   \\
& ~\le ~ c_{\phi,\psi}\phi(1 )+ \log(u) +1  \\
&~\le~ (c_{\phi,\psi}+1 ) \tilde{\phi}(u).
\end{align*}
Summarizing all, we obtain
$$\int_0^ u \frac{\tilde{\psi}(t)}{t}dt +u\int_u^\infty \frac{\tilde{\psi}(t)}{t^2}dt \le (c_{\phi,\psi}+1 )\tilde{\phi}(u), ~u >0.$$
Thus, all the assumptions of Theorem \ref{thstz2} hold for the functions $\tilde{\phi}$ and $\tilde{\psi}$.
Hence, in particular, the Calder\'{o}n operator $S_\infty :\Lambda_{\tilde\phi}(0,\infty )\to\Lambda_{\tilde\psi}(0,\infty ) $ is bounded. Therefore, for any
$z\in \Lambda_\phi\subset \Lambda _{\tilde \phi}(0,\infty )$
(we extend $z$ to a function on $(0,\infty)$ by setting $z=0$ on $(1,\infty)$), we have
$$   \left\| Sz \right\|_{\Lambda_{\psi}}  \leq \left\|S_{\infty}z \right\|_{  \Lambda_{\tilde{\psi}}(0,\infty)}\le C\|z\|_{  \Lambda_{\tilde{\phi}}(0,\infty)}=C\|z\|_{  \Lambda_{\phi}},$$
which implies that $S$ is bounded from $\Lambda_{\phi}$ in $\Lambda_{\psi}$.

To prove (ii), we take $x\in \Lambda_{\psi}\subset \Lambda_{\tilde\psi}(0,\infty )$. By Theorem \ref{thstz2}, there is $y\in \Lambda_{\tilde\phi}(0,\infty )$ such that $\mu(x)\leq S_\infty \mu(y)$ and
\begin{align}\label{ylexpsi}
\left\|y\right\|_{\Lambda_{\tilde\phi}(0,\infty)}
\leq 8\left\|x\right\|_{\Lambda_{\tilde\psi}(0,\infty )}.
\end{align}
Without loss of generality, we may assume that $y=\mu(y).$ Then, if
$$
z(t):=\left(y(t)+\int_1^{\infty}y(s)\,\frac{ds}{s}\right)\chi_{_{(0,1)}}(t),\quad t\in(0,1),$$
we have
$$\mu(t;x)\leq (S_\infty y)(t) \leq (Sz)(t),\quad t\in(0,1).$$
On the other hand,
\begin{align*}
\left\|z\right\|_{\Lambda_{\phi}}
&~\leq~ \left\|
y\chi_{_{(0,1)}}\right\|_{\Lambda_{\phi}}
+\int_1^{\infty}y(s)\frac{ds}{s} \\
&~=~  \left\|
y  \right\|_{\Lambda_{\tilde\phi}(0,\infty )}\\
&\stackrel{\eqref{ylexpsi}}{\leq}  8 \left \|x \right\|_{\Lambda_{\tilde \psi}(0,\infty )} \\
&~=~ 8  \left\|x\right\|_{\Lambda_{\psi}}.
 \end{align*}
This completes the proof of the theorem.
\end{proof}

Recall (see Corollary \ref{opt-gen}) that the space $(\cS_{\Lambda_{\phi}}, \left\|\cdot \right \|_{ \cS_{\Lambda_{\phi}}})$ is   the optimal symmetric quasi-Banach range for the martingale transform  $T$ on the space $\Lambda_{\phi}$.
The following result shows that, under the assumptions of Theorem~\ref{main th}, it can be identified as the Lorentz space $\Lambda_{\psi}$ from this theorem.
\begin{corollary}\label{main th2} If the assumptions of Theorem \ref{main th} hold,  then we have $\cS_{\Lambda_{\phi}} =\Lambda_{\psi}.$   Thus, $\Lambda_{\psi}$ is the optimal symmetric (quasi-)Banach range for the martingale transform  $T$ defined on the Lorentz space $\Lambda_{\phi}$.
\end{corollary}
\begin{proof}
First, by Theorem \ref{main th} (i), $S:\,\Lambda_{\phi}\to\Lambda_{\psi}$ is a bounded operator.
Hence, it follows from Theorem \ref{optimal01}  that
$\cS_{\Lambda_{\phi}}\subset\Lambda_{\psi}.$

To prove the converse inclusion, we assume that  $x\in \Lambda_{\psi}$. Then, by Theorem \ref{main th} (ii) there exists $y\in\Lambda_{\phi}$ such that $\mu(x)\leq S\mu(y).$ Hence, from the definition of the space $\cS_{\Lambda_{\phi}}$ it follows $x\in  \cS_{\Lambda_{\phi}}$, and we conclude that  $\Lambda_{\psi}\subset \cS_{\Lambda_{\phi}}.$
\end{proof}

\subsection{Proof of Corollary \ref{maincor}}\label{3}
\label{A1}


 \begin{proof}[Proof of Corollary  \ref{maincor}]

Let $E \subset L_1 $  and $F$ be quasi-Banach symmetric function spaces on $(0,1)$. From the estimates obtained in Theorem \ref{main thm} it follows that assertions $(1)$ and $(3)$ are  equivalent. Moreover, it is well known that the Hilbert transform
$H$ is bounded from $E$ in $F$ if and only if so is $S:E\to F$
(see e.g. \cite[Theorems 3.6.8 and 3.6.10]{BSh} and   the classical result  \cite[Theorem 2.1]{B}).
Therefore, we obtain that $(1)\Longleftrightarrow (2)\Longleftrightarrow (3)$. It remains to prove implications $(3)\Rightarrow (4)$ and $(4)\Rightarrow (1)$ whenever $E$ is separable.

$(3)\Longrightarrow (4)$. Let $A\subset \mathbb{N}$.
 The operator $P_A$ is a martingale transform with respect to the Haar filtration and so, by \cite{Bur} (see also \cite[Theorem~3.3.7]{KS}, \cite[II p.156]{LT} or \cite{Pisier}), $P_A$ can be extended to a bounded linear operator from $ L_1$ into $ L_{1,\infty}$ with norm which does not depend on the set $A$. Therefore (see \eqref{above estimate} and subsequent references), we have
$$\mu(P_Ax)\leq c_{{\rm abs}}S\mu(x),\quad x\in L_1,\quad A\subset\mathbb{N}.$$
Hence,
$$\left\| P_Ax \right \|_{F} \leq c_{\rm abs}\left\|S\mu(x)\right\|_{F}\le c_{\rm abs}\left\|S\right\|_{E\to F} \left\|x\right\|_{E},\quad x\in E,\quad A\subset\mathbb{N}.$$

Finally, observe that implication $(4)\Longrightarrow (1)$ follows from the fact that
$$T=P_A,\quad A=\{1\}\bigcup\Big(\bigcup_{n\geq1}\{2^{2n-1}+1,\cdots,2^{2n}\}\Big).$$
\end{proof}

Recall that any Lorentz space $\Lambda_\phi$ on $[0,1]$, with $\phi(+0)=0$, is separable (see e.g. \cite[Lemma~II.5.1]{KPS}). Thus, the next result follows immediately from Corollary  \ref{maincor}. It complements results in  \cite{STZ2} (see also the motivation provided in \cite[section 4]{B}).

\begin{corollary}
Let the assumptions of Theorem \ref{main th} hold. The following statements are equivalent:
\begin{enumerate}
\item The martingale transform $T$ is bounded from $\Lambda_\phi$ into $\Lambda_\psi$.
\item The Hilbert transform  $H$ is bounded from $\Lambda_\phi$ into $\Lambda_\psi$.
\item The Calder\'{o}n operator $S$ is bounded from $\Lambda_\phi$ into $\Lambda_\psi$.
\item Every Haar basis projection is bounded from $\Lambda_\phi$ into $\Lambda_\psi$.
\end{enumerate}
\end{corollary}

\subsection{Narrow operators}\label{s4}
 Let $E$ be a quasi-Banach  symmetric function space on $[0,1]$ and let $X$ be an $F$-space \cite{KPR,Popov-Randri}.
A bounded linear operator $T:E\to X$ is called \emph{narrow} if for each set $A\subset  (0,1)$ and arbitrary $\varepsilon >0$ there exists a sign $x$ on $A$ (i.e., $x$ is a function supported on $A$ and taking values in the set $\{-1,1\}$ on $A$) such that $\norm{Tx}_X< \varepsilon$ \cite[Proposition 1.9(ii)]{Popov-Randri}.

It is well known \cite{Popov-Randri,Plichko-Popov,Popov-2011} that the identity operator on a separable symmetric space $E$ is a sum of two narrow operators bounded on $E$ whenever $E$ has an unconditional basis (equivalently, $E$ is an interpolation space between $L_p$ and $L_q$ for some  $1<p<q<\infty$ \cite[II. p.161]{LT}).
 The main result in this subsection  is linked with the following open problem stated in \cite{PSV}:

 {\it Assume that the identity operator $id$ on a separable symmetric space $E$ on $(0,1)$ may be represented as a sum of two narrow operators bounded on $E$. Does this imply that $E$ has an unconditional basis?}


In Theorem~\ref{t2} below, we show that the identity operator on any separable quasi-Banach symmetric function space $E$ such that $E\subset L_1$ is a sum of two narrow operators (basis projections), which are bounded from $E$ into the optimal range $\mathcal{S}_E$ of the  Calder\'{o}n operator $S$ on $E$ (see Section \ref{S}). This extends the above-mentioned result for  symmetric function spaces   having non-trivial Boyd indices.


\begin{theorem}\label{t2}  If $E$ be a separable quasi-Banach symmetric function space on $(0,1)$ with $E\subset L_1,$ then
the identity operator ${\rm id}:E \to  E$ is a sum of two narrow operators bounded from $E$ into $\mathcal{S}_E$.
\end{theorem}


Let $T$ be the operator defined  in \eqref{T-0}.
We write
$${\rm id}=T+({\rm id}-T).$$
To prove Theorem \ref{t2}, it suffices to prove the following lemma.
\begin{lem} Let $E$ be a separable quasi-Banach symmetric function space on $(0,1)$ with $E\subset L_1.$
Then, the operators $T,{\rm id}-T:\,E\to\mathcal{S}_E$ are narrow.
\end{lem}
\begin{proof} As above, $h_{n,k}$'s are Haar functions. Recall that
$$Th_{n,k}=
\begin{cases}
h_{n,k},& n\mbox{ is odd}\\
0,& n\mbox{ is even}
\end{cases},\quad
({\rm id}-T)h_{n,k}=
\begin{cases}
h_{n,k},& n\mbox{ is even}\\
0,& n\mbox{ is odd}
\end{cases}.
$$

We only prove the assertion for  the operator $T$   as the argument for $ {\rm id}-T$ follows {\it mutatis mutandi}.

%

 Since $E$ is separable, it follows from  \cite[Lemma 1.12]{Popov-Randri} that it
suffices to prove that for any dyadic interval $\Delta_m^l=[\frac{l-1}{2^m},\frac{l}{2^m})$ for any $m=0,1,\cdots$ and $l=1,\cdots, 2^m$, there exists $x \in E$ with $x^2=\chi_{_{\Delta _m^l}}$ and $Tx=0$. 

Observe that
$$\Delta_m^l=\Delta_{m+1}^{2l-1}+\Delta_{m+1}^{2l} = \Delta_{m+2}^{4l-3}+\Delta_{m+2}^{4l-2}+\Delta_{m+2}^{4l-1}+\Delta_{m+2}^{4l}.$$
Letting
$$x=\begin{cases}h_{m,l}, &\mbox{if $m$ is even};\\
 h_{m+1,2l-1}+ h_{m+1,2l}, &\mbox{if $m$ is odd},
\end{cases}
$$
we have $x^2 =\chi_{_{\Delta_m^l}}$ and $Tx=0.$ This completes the proof.
\end{proof}


\begin{thebibliography}{99}


\bibitem{Astashkin20}
S.   Astashkin, {\it The Rademacher system in function spaces}, Birkhauser, 2020.

\bibitem{AstMil21}
S.   Astashkin, M. Milman, {\it Extrapolation Stories and Problems,} Pure and Appl. Funct. Anal. \textbf{6} (2021), No. 3, 651--707.



\bibitem{Banuelos}
R. Banuelos,
{\it The foundational inequalities of D.L. Burkholder and some of their ramifications,}
Illinois J. Math. \textbf{54} (2012), 789--868.


\bibitem{BSh} C. Bennett, R. Sharpley, {\it Interpolation of Operators}, Pure and Applied Mathematics, {\bf 129}. Academic Press, 1988.



\bibitem{B} D. Boyd, {\it The Hilbert transform on rearrangement-invariant spaces}, Can. J. Math., {\bf 19} (1967), 599--616.

\bibitem{Bur}
D. Burkholder, {\it Martingale transforms}, Ann. Math. Statist. \textbf{37} (1966), 1494--1504.

\bibitem{Burkholder-82} D. Burkholder,  {\it  A nonlinear partial differetial equation and the unconditional constant of the Haar system in $L^p$}, Bull. (New Ser.) AMS, 1982, Vol. 7, No. 3, 591--595.

\bibitem{Burkholder-84} D. Burkholder,  {\it  Boundary value problems and sharp inequalities for martingale transforms}, Ann. Probab. {\bf 12} (3) (1984),  647--702.


 \bibitem{Burkholder-01} D. Burkholder,  {\it   Martingales and singular integrals in Banach spaces,}Handbook of the geometry of Banach spaces, Vol. I, 233--269, North-Holland, Amsterdam, 2001.

\bibitem{cal} A. P. Calder\'{o}n, \textit{Spaces between }$L^{1}$\textit{ and }$L^{\infty}$\textit{\ and the theorem of Marcinkiewicz}, Studia Math. \textbf{26} (1966), 273--299.

\bibitem{HLP}
G.   Hardy, J.  Littlewood,   G. Polya. {\it Inequalities}, Cambridge University
Press, London, second edition, 1952.











\bibitem{HSZ}
J. Huang, F. Sukochev, D. Zanin,
{\it Optimal estimates for martingale transforms,}
J. Funct. Anal. (2022).


\bibitem{HNVW}
T. Hytonen, J. van Neerven, M. Veraar, L. Weis,
{\it Analysis in Banach spaces, vol. I
Martingales and Littlewood--Paley theory},
 Ergebnisse der Mathematik und ihrer Grenzgebiete. 3. Folge. A Series of Modern Surveys in Mathematics, 63,
Springer, Cham, 2016.

\bibitem{jm} B.~Jawerth, M.~Milman, \textit{Extrapolation theory with
Applications}, Mem. Amer. Math. Soc. 1991, V.~89, No.~440.

\bibitem{JSWZ}
Y. Jiao, F. Sukochev, L. Wu, D. Zanin, {\it
Distributional inequalities for non-commutative martingales},  arxiv.org/abs/2103.08847.
\bibitem{KPR}
N. Kalton, N. Peck, J. Rogers, {\it An F-Space Sampler,} London Math. Soc. Lecture Note Ser., vol.89, Cambridge University Press, Cambridge, 1985.

\bibitem{KS}
B.  Kashin,  A. Saakyan, {\it Orthogonal series,} Amer. Math. Soc., Providence
RI, 1989.



\bibitem{KPS} S. Krein, Y. Petunin, and E. Semenov, {\it Interpolation of linear operators}, Amer. Math. Soc., Providence, R.I., (1982).

\bibitem{LSU}
O.  Lelond, E.  Semenov, S.   Uksusov,
{\it The space of Fourier--Haar
multipliers,}
 Siberian Math. J. \textbf{46} (2005), 103--110.

\bibitem{LT} J. Lindenstrauss, L. Tzafiri, {\it Classical Banach spaces}. Springer-Verlag, I, II, (1979).



	
\bibitem{NS}
I. Novikov, E.  Semenov, {\it Haar series abd linear operators}, Kluwer Acad. Publ. Amsterdam, 1997. 	

\bibitem{Paley}
R. Paley,
{\it A remarkable series of orthogonal functions,}
Proc. London Math. Soc. {\bf 34} (1932), 241--264.


\bibitem{Pisier}
G. Pisier,
{\it Martingales in Banach spaces,}
Cambridge studies in advanced mathematics 155, Cambridge University Press, 2016.


\bibitem{Plichko-Popov} A. Plichko, M. Popov, {\it Symmetric function spaces on atomless probability spaces.} Dissertationes Math. (Rozprawy Mat.) {\bf 306} (1990), 85 pp.

\bibitem{Popov-2011}M. Popov,  {\it Narrow operators (a survey).} Function spaces IX, 299--326, Banach Center Publ., 92, Polish Acad. Sci. Inst. Math., Warsaw, 2011.

\bibitem{Popov-Randri} M. Popov, B. Randrianantoanina, {\it Narrow operators on function spaces and vector lattices.} De Gruyter Studies in Mathematics {\bf 45},  Berlin, 2013.

\bibitem{PSV} M. Popov, E. Semenov, D. Vatsek, {\it Some problems on narrow operators on function spaces.} Cent. Eur. J. Math. {\bf 12} (2014), no. 3, 476--482.



\bibitem{Francia}
J. Rubio de Francia,
{\it Martingale and integral transforms of Banach space valued functions,}
in Probability and Banach spaces (Zaragoza, 1985). Lecture Notes in Math., 1221, Springer, Berlin, 1986, pp. 195--222.

\bibitem{semen} E. M. Semenov, {\it Estimates for operators of weak type}, Funct. anal. \& related topics (Sapporo, 1990), 172--178, World Sci. Publ., River Edge, NJ, 1991.

\bibitem{Sh-80} R. Sharpley, {\it Counterexamples for classical operators on Lorentz-Zygmund spaces}, Studia Math. {\bf 68} (1980),  141--158.















\bibitem{STZ} F.  Sukochev, K. Tulenov, D. Zanin, {\it The optimal range of the Calder\' on operator and its applications.} J. Funct. Anal. {\bf 277} (10) (2019),  3513--3559.

 \bibitem{STZ2} F.  Sukochev, K. Tulenov, D. Zanin,   {\it
 The boundedness of the Hilbert transformation from one rearrangement invariant Banach space into another and applications},
 Bull. Sci. Math. {\bf 167} (2021), 102943.




\end{thebibliography}
\end{document}